\documentclass[12pt, reqno]{amsart}

\usepackage{fourier}
\usepackage{nicefrac}

\usepackage[T1]{fontenc}
\usepackage[latin9]{inputenc}
\usepackage{amsthm}
\usepackage{amstext}
\usepackage{amssymb}

\usepackage{tikz}
\usetikzlibrary{positioning}

\usepackage[margin=1.5in]{geometry}

\usepackage{graphicx}
\usepackage{color}
\usepackage{url}

\usepackage{amsmath,amsfonts,latexsym,amsthm,amstext,enumerate,color}
\usepackage[numeric,initials,nobysame]{amsrefs}

\numberwithin{equation}{section}

\newcommand{\ca}[1]{{#1}}

\newtheorem{theorem}{Theorem}[section]
\newtheorem*{theorem*}{Theorem}
\newtheorem{lemma}[theorem]{Lemma}

\newtheorem{proposition}[theorem]{Proposition}

\newtheorem*{observation*}{Observation}

\theoremstyle{definition}{

}

\newcommand{\R}{\mathbb R}

\newcommand{\eps}{\varepsilon}

\renewcommand{\and}{\hbox{ {\rm and} }}

\newcommand{\E}{\mathbb{E}}
\renewcommand{\P}{\mathbf{P}}
\newcommand{\prob}{\P}

\def\capp{{\rm cap}}

\renewcommand{\epsilon}{\varepsilon}

\newcommand{\be}{\begin{eqnarray}}
\newcommand{\ee}{\end{eqnarray}}

\date{}

\AtEndDocument{\bigskip{\footnotesize%
\par
\textsc{Ori Gurel-Gurevich} \par
\textsc{Einstein Institute of Mathematics, Hebrew University of Jerusalem} \par
\textit{Email:} \texttt{origurel@math.huji.ac.il} \par
\addvspace{\medskipamount}
\textsc{Asaf Nachmias} \par
\textsc{Department of Mathematics, University of British Columbia and} \par \textsc{School of Mathematical Sciences, Tel Aviv University} \par
\textit{Email:} \texttt{asafnach@math.ubc.ca} \par
\addvspace{\medskipamount}
\textsc{Juan Souto} \par
\textsc{IRMAR, Universite de Rennes 1}\par
\textit{Email:} \texttt{juan.souto@univ-rennes1.fr} \par
}}

\begin{document}
\title{Recurrence of multiply-ended planar triangulations}

\author{Ori Gurel-Gurevich \quad Asaf Nachmias \quad Juan Souto}
%

%
%
%

\def\sX {{\mathbb X}}
\def\sE {{\mathcal E}}

\begin{abstract} In this note we show that a bounded degree planar triangulation is recurrent if and only if the set of accumulation points of some/any circle packing of it is polar (that is, planar Brownian motion avoids it with probability $1$).  This generalizes a theorem of He and Schramm \cite{HS} who proved it when the set of accumulation points is either empty or a Jordan curve, \ca{in which case the graph has one end}. We also show that this statement holds for any straight-line embedding with angles uniformly bounded away from $0$.
\end{abstract}

\maketitle



\vspace{-1cm}

\section{Introduction}

A {\bf circle packing} of a planar graph $G=(V,E)$ is a set of circles $\{C_v\}_{v\in V}$ in the plane with disjoint interiors such that $C_v$ is tangent to $C_u$ if and only if $u$ is adjacent to $v$ in $G$. Koebe's famous circle packing theorem \cite{K36,St} asserts that any {\em finite} planar graph has a  circle packing; furthermore, when $G$ is a triangulation this packing is unique up to M\"obius transformations and reflections of the plane.

In their seminal paper, He and Schramm \cite{HS} studied the analogue theory for {\em infinite} planar graphs and found inspiring connections between parabolicity or hyperbolicity of the circle packing and recurrence or transience of the corresponding random walk on $G$.  Their theory is restricted to one-ended planar triangulations, that is, the removal of any finite set of vertices of $G$ and the edges adjacent to it results in some finite components and a single infinite component. The goal of this note is to generalize their theory to any bounded degree planar triangulation which can be multiply-ended.

Given a circle packing $P=\{C_v\}$, a point $z \in \R^2$ is called an {\bf accumulation point} of $P$ if any neighbourhood of $z$ intersects infinitely many circles of $P$. A measurable subset $A \subset \R^2$ is called {\bf polar} if planar Brownian motion $\{B_t\}_{t \geq 0}$ avoids it with probability one, that is, the event $B_t \not \in A$ for all $t >0$ occurs almost surely. The simple random walk on $G$ is said to be {\bf recurrent} if it returns to its starting vertex almost surely.

\begin{theorem} \label{mainthm}
If $P=\{C_v\}$ is a circle packing of a bounded degree planar triangulation $G$, then the simple random walk on $G$ is recurrent if and only if the set of accumulation points of $P$ is a polar set.
\end{theorem}

\subsection{Extensions} It will be useful to consider more general embeddings of planar graphs. Given a planar graph $G$, an {\bf embedding with straight lines} of it in the plane is a map \ca{taking} vertices to distinct points in $\R^2$ and edges to the straight lines between the corresponding vertices such that no two edges cross. A point $p \in \R^2$ is called an {\bf accumulation point} of an embedding of $G$ if any neighbourhood of it intersects infinitely many edges. An embedding with straight lines of a planar triangulation $G$ is {\bf good} if the angles between any two adjacent edges are uniformly bounded away from $0$. We note that this also implies that the angles are uniformly bounded away from $\pi$ and, by the sine law, the ratio between the lengths of two adjacent edges is uniformly bounded. \\
%
%
%
%
%

\noindent {\bf Theorem 1.1'} If $G$ \ca{is} a bounded degree planar triangulation with a good embedding, then the simple random walk on $G$ is recurrent if and only if the set of accumulation points of the embedding is a polar set. \\

Given a circle packing $P=\{C_v\}$ of a bounded degree planar triangulation $G$ we may obtain a straight line embedding by mapping vertex $v$ to the center of $C_v$. Rodin and Sullivan's ring lemma \cite{RS} asserts that the ratio between radii of tangent circles are bounded by a constant depending only on the degree. Hence, in each such triangle the altitude's length is comparable to the edge length and so the angles are uniformly bounded away from $0$. We deduce that the straight lines embedding obtained from $P$ is a good embedding and so Theorem 1.1' implies Theorem 1.1.

Our last extension of Theorem \ref{mainthm} handles planar graphs with bounded degree that are not necessarily triangulations but have a bounded number edges on each face. The embedding condition presented in the next theorem appeared in \cite{Chelkak}, see also \cite{ABGN}.
\begin{theorem}\label{fromabgn} Let $G$ be a bounded degree planar graph embedded in the plane with straight lines such that
\begin{itemize}
\item[(a)] For each face, all the inner angles are bounded away from $\pi$ uniformly (so in particular all faces are convex, there is no outer face, and each face has a bounded number of edges).
\item[(b)] The length of adjacent edges is comparable.
\end{itemize}
Then the simple random walk on $G$ is recurrent if and only if the set of accumulation points is a polar set.
\end{theorem}
\begin{proof}
Construct a new planar triangulation $G'$ by \ca{adding a new vertex for each face of $G$ and connecting it to all the vertices of that face. Embed this graph by taking the embedding of $G$ and map each new vertex of $G'$ to the center of mass of the vertices of its corresponding face.}
It is immediate that the set of accumulation points of $G$ and $G'$ is the same. Let us also observe that this is a good embedding of $G'$.  Indeed, \cite[Lemma 2.1]{ABGN} asserts that the angle between any two adjacent edges is bounded away uniformly from $0$. Hence, if $u',u,v,v'$ are consecutive vertices in a face of $G$ (it is possible that $u'=v'$ when the face is a triangle), then the distance between $u'$ and the infinite straight line connecting $u$ and $v$ is comparable to the edge length of $uv$, and likewise for $v'$. By convexity, the same holds for any other vertex in the face and therefore the distance between the centre of mass and each edge of the face is comparable to the edge length of the face. It follows \ca{that} the angles of triangles in $G'$ are bounded away from $0$.

 Furthermore, $G$ and $G'$ are immediately seen to be {\em roughly equivalent}, see \cite[Section 2.6]{LP}. Thus, Theorem 1.1' applied to $G'$ combined with \cite[Theorem 2.17]{LP} gives the required result.
\end{proof}

\noindent{\bf Remark.} As the proof above suggests, condition (a) in Theorem \ref{fromabgn} may be replaced by the condition that each face can be triangulated into a bounded number of triangles such that each triangle has angles bounded away from $0$. In particular, the faces need not be convex.

\subsection{About the proof} The proofs by He and Schramm \cite{HS} are based on the discrete analogue of the notion of {\em extremal length} and on the classical fact that the simple random walk on a graph is recurrent if and only if the edge extremal length from any finite set to infinity is infinite. It is assumed in \cite{HS} that the graph $G$ is one-ended (what they call {\em disk triangulations})
and the proofs do not seem to generalize to the multiply-ended case.

We present here a very short proof of Theorem \ref{mainthm}. In particular, we do not use the theory of extremal length or quasi-conformal maps, however, the heuristic leading to the proof relies on these notions, so let us briefly describe it.

Let $G$ be a planar triangulation with a good embedding. Given $G$ we may construct a Riemannian surface $S(G)$ by considering each face as an equilateral triangle of unit length and gluing them together according to the combinatorics of the graph, as done in \cite{RG}. A curve in $S(G)$ has a length associated with it which is obtained by adding up the Euclidean length of the curve in each \ca{equilateral} triangle it passes through. Thus, the natural shortest-length metric on $S(G)$ is defined. It is easy to verify that this metric on $S(G)$ is roughly isometric to the graph-distance metric on $G$.

Next, consider the injective map $\varphi : S(G) \to \R^2 \setminus A$ where $A$ is the set of accumulation points of the packing obtained by mapping the vertices of $S(G)$ to the centers of circles of the packing $P=\{C_v\}$ and then linearly interpolating so that the map between each unilateral triangle of $S(G)$ is mapped affinely to the corresponding triangle formed by the circle centers of the face. Since $G$ has bounded degrees, Rodin and Sullivan's ring lemma \cite{RS} implies that the angles in each such triangle in $\R^2 \setminus A$ are bounded away from $0$ and $\pi$ uniformly, thus the map $\varphi$ is $K$-quasi-conformal (see \cite{Ahlfors}) with dilatation $K$ depending only on the bound on the degree.

It is a classical fact due to Kanai \cite{Kanai} (see also \cite[Theorem 2.17]{LP}) that parabolicity of a surface or recurrence of a graph (see precise definitions below) is invariant to rough isometries, so $G$ is recurrent if and only if $S(G)$ is parabolic. Another classical fact is that quasi-conformal maps change extremal length by at most a multiplicative factor depending only on $K$ (see \cite{Ahlfors}) and since $\varphi$ was quasi-conformal, $S(G)$ is parabolic if and only if $\R^2 \setminus A$ is parabolic, concluding the proof.

\section{The Proof}
\subsection{Recurrence and parabolicity of graphs and surfaces}
\subsubsection{Recurrence}
Given a finite set of vertices $K \subset V$, the {\em capacity of $K$} is defined by
$$ \capp(K) = \inf_{u} \sum_{e=(x,y)\in E} |u(x)-u(y)|^2 \, ,$$
where $u : V \to \R$ has finite support and satisfies $u(v)=1$ for any $v\in K$. It is easy to see that in any connected graph if the capacity of a finite set $K$ is positive, then the capacities of all finite sets are positive. The following is a classical result (see \cite[Exercise 2.13]{LP}).


\begin{proposition} \label{capacityrecurrence} A connected graph $G=(V,E)$ is recurrent if and only if $\capp(K)=0$ for some/all finite sets $K\subset V$.
\end{proposition}

\subsubsection{Parabolicity} Let $S \subset \R^2$ be a domain and consider Brownian motion on it with Dirichlet boundary conditions (that is, the motion \ca{is killed} upon hitting $\partial S$). We say that $S$ is parabolic if for any open set $U \subset S$ Brownian motion started at any point visits $U$ almost surely. For example, $\R^2$ minus an isolated set of points is parabolic, but the $\R^2 \setminus [0,1]$ is not.  As in the discrete case, there is a condition for parabolicity in terms of capacities. The {\em capacity} of a compact set $K$ is defined by
$$ \capp(K) = \inf _{u} \int_S | \bigtriangledown u |^2 d \mu \, ,$$
where the infimum is taken over all Lipschitz functions $u:S\to \R$ with compact support such that $u_{|K}=1$ and $0 \leq u \leq 1$. Again, it is easy to see that if the capacity of one compact set is positive, then all compact sets have positive capacities. The following is the continuous analogue of Proposition \ref{capacityrecurrence} (see \cite[Chapter 5]{Grigoryan} for the proof and a general discussion of Brownian motion on Riemannian manifolds).
\begin{proposition} \label{capacityparabolic} A domain $S$ is parabolic if and only if $\capp(K)=0$ for some/all compact sets $K\subset S$.
\end{proposition}


\subsection{Proof of Theorem 1.1'}

Let us first prepare some notation. Given a bounded degree triangulation $G$ and a good embedding of it, we abuse notation by referring to $v$ both as a vertex in $G$ and as a point in the plane given by the embedding and to $e$ both as an edge in $G$ and as corresponding straight line in the plane. We denote by $|e|$ the Euclidean length of $e$ and \ca{let $S\subset \R^2$ be the domain obtained by taking the union over all the faces of the closed polygon inscribed by the face. By definition $\partial S = A$,} where $A$ is the set of accumulation points. Lastly, since the embedding of $G$ is good, we denote by $\eta>0$ the lower bound on the angles between adjacent edges.

We begin by showing that if $G$ is recurrent, then $S$ is parabolic. Fix a vertex $\rho$ of $G$ and let $K=\{\rho\}$. Since $G$ is recurrent, by Proposition \ref{capacityrecurrence} we have that $\capp(K)=0$. Let $\eps>0$ be arbitrary and let $u:V(G) \to \R$ be a function with finite support, $u(\rho)=1$ and $\sum_{(x,y)\in E} |u(x)-u(y)|^2 < \eps$. Define $\tilde{u}:S\to \R$ to be the linear interpolation of $u$ on $S$, that is, if $\xi\in S$ is inside the triangle $v_1 v_2 v_3$ and $\xi = x v_1 + y v_2 + z v_3$ for $x,y,z \geq 0$ and $x+y+z=1$, then $\tilde{u}(\xi) = x u(v_1) + y u(v_2) + z u(v_3)$. Note that $\tilde{u}$ is Lipschitz and compactly supported.

We claim that there exists a constant $C=C(\eta) < \infty$ such that if $\xi$ is inside the triangle $v_1 v_2 v_3$, then
\be\label{gradbound} |\bigtriangledown \tilde{u} (\xi)| \leq C |v_1 v_2|^{-1} \ca{\max_{(i,j) \in \{1,2,3\}^2} |u(v_i) - u(v_j)|} \, .\ee

\noindent Indeed, assume without loss of generality that $\min_{i=1,2,3} u(v_i) = u(v_1)$ and by subtracting a $u(v_1)$ from $u$ we may further assume that $u(v_1)=0$ and that $u(v_2)$ and $u(v_3)$ are non-negative. For any $\xi$ inside the triangle $v_1 v_2 v_3$ we have that ${d \tilde{u} \over d y}(\xi) = |v_2 v_1|^{-1} u(v_2)$ and ${d \tilde{u} \over d z}(\xi) = |v_3 v_1|^{\ca{-1}} u(v_3)$. Since the edge lengths of the triangle $v_1 v_2 v_3$ are comparable, \eqref{gradbound} follows.


Since the embedding is good, $|v_1 v_2|^2$ is comparable to the area of the triangle $v_1 v_2 v_3$ (with a constant depending only on $\eta$), thus,
$$ \int _{v_1v_2v_3} |\bigtriangledown \tilde{u}|^2 d\mu \leq C \max_{(i,j) \in \{1,2,3\}^2} |u(v_i) - u(v_j)|^2 \, ,$$
and hence $\int _S |\bigtriangledown \tilde{u}|^2 d\mu \leq C \sum_{e=(x,y)\in E} |u(x)-u(y)|^2 \leq C\eps$ and we conclude by Proposition \ref{capacityparabolic} that $S$ is parabolic. \\

\begin{figure}
\begin{center}
\includegraphics[width=0.5\textwidth]{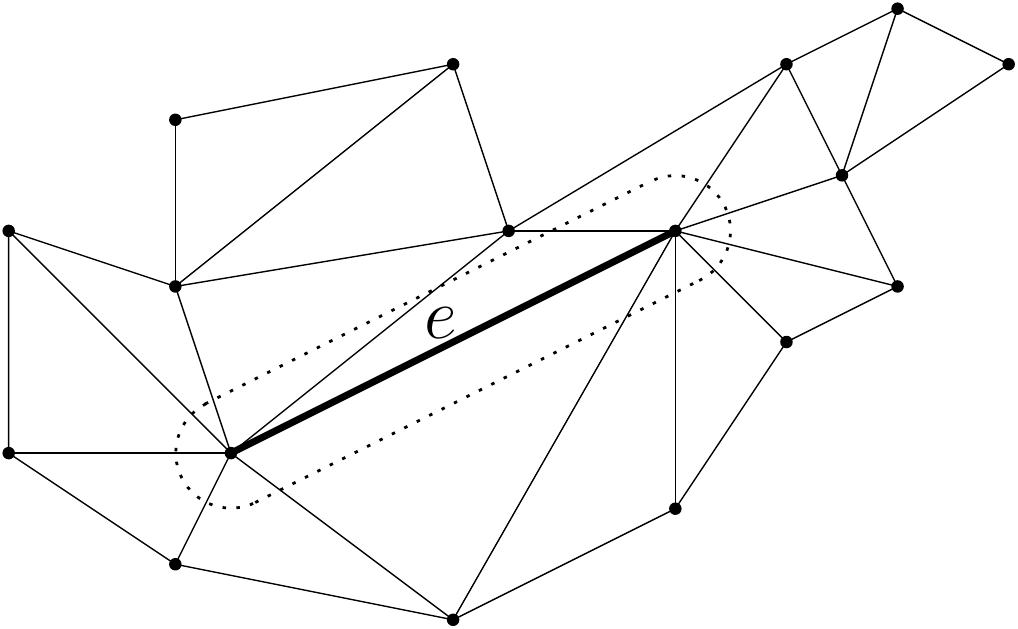}
\caption{The sausage lemma: no edge that is not adjacent to $e$ can intersect the marked ``sausage'' around $e$
  with width $c|e|$.}
\label{fig.sausage}
\end{center}
\end{figure}

To prove the converse, let us first recall an easy geometric estimate of \cite{ABGN}.
\begin{lemma}[Sausage lemma \cite{ABGN}] \label{sausage} There exists $c=c(\eta)>0$ such that if $e,f$ are non adjacent edges, then $d(e,f) \geq c|e|$, where $d(\cdot,\cdot)$ is Euclidean distance.
\end{lemma}
Now, assume $S$ is parabolic and let $K \subset S$ be a compact set so that $\capp(K)=0$ by Proposition \ref{capacityparabolic}. Let $\eps>0$ be arbitrary and let $\tilde{u}:S \to \R$ be a Lipschitz function compactly supported such that $u_{|K}=1$ and $\int_S |\bigtriangledown \tilde{u}(z)|^2 dz \leq \eps$. We define a function $u: V \to \R$ as follows. Given a vertex $x$, let $r_x$ be the minimal length of an edge adjacent to $x$ and let $X$ be a uniform random point in the Euclidean ball around $x$ of radius $c r_x$, where $c$ is the constance from Lemma \ref{sausage}. We define $u(x) = \E \tilde{u}(X)$. Let $(x,y)$ be an edge in $G$;  by Jensen's inequality
$$ (u(x) - u(y))^2 \leq \E \big [ (\tilde{u}(X) - \tilde{u}(Y))\ca{^2} \big ] \leq \E \Big [ \Big ( \int_{\Gamma_{XY}} |\bigtriangledown \tilde{u} (z) | dz \Big )^2 \Big ] \, ,$$
where for any two points $p,q$ in $\R^2$ we denote by $\Gamma_{pq}$ the straight line between them. Note that we used Lemma \ref{sausage} to assert that $\Gamma_{XY} \subset S$. We now use Cauchy-Schwartz and the fact that $|\Gamma_{XY}| \leq C r_x$ for some $C=C(\eta)<\infty$ to bound
$$ |u(x) - u(y)|^2 \leq C r_x \E \Big [ \int_{\Gamma_{XY}} |\bigtriangledown \tilde{u}(z)|^2 dz \Big ] \leq C r_x^2 \E \big [ |\bigtriangledown \tilde{u}(Z)|^2 \big ] \, ,$$
where, conditioned on $X,Y$ the point $Z$ is drawn uniformly on the line $\Gamma_{XY}$. We are left to bound the Radon-Nikodym derivative of the law of $Z$ with respect to Lebesgue measure. Indeed, let $T_{xy}$ be the union of the triangles touching $x$ or $y$ and let $f_Z$ be the density of $Z$. It is clear that $f_Z=0$ out of $T_{xy}$. Given any $z \in T_{xy}$ let $B(z,\delta)$ be the Euclidean ball of radius $\delta$ around $z$. We claim that for some $C=C(\eta)<\infty$
\be\label{toprove} \prob( Z \in B(z,\delta) ) \leq C  \delta^2 r_x^{-2} \, .\ee
Indeed, assume without loss of generality that $z$ is closer to $x$ than to $y$ and condition on the value of $Y$. Then for the event $Z \in B(z,\delta)$ to occur we must have that $X$ belongs to the cone emanating from $Y$ that is tangent to the circle $\partial B(z,\delta)$ --- this has probability at most $Cr^{-1}\delta$. Conditioned on this event, $Z$ must fall on an interval of length at most $2\delta$ on the line $\Gamma_{XY}$ --- this has probability at most $Cr^{-1} \delta$, showing \eqref{toprove}. Hence $f_Z(z) \leq C r_x^{-2}$ for any $z\in T_{xy}$. Since the area of $T_{xy}$ is proportional to $r_x^2$ we get that
$$ |u(x) - u(y)|^2 \leq C \int_{T_{xy}} |\bigtriangledown \tilde{u}(z)|^2 dz \, .$$
We sum this inequality over all edges $(x,y)$ in $G$ and since the degree of $G$ is bounded we integrate over each triangle a bounded number of times, whence
$$ \sum_{(x,y) \in E(G)} |u(x)-u(y)|^2 \leq C \int_S |\bigtriangledown \tilde{u}(z)|^2 dz \leq C\eps\, ,$$
showing that $G$ is recurrent by Proposition \ref{capacityrecurrence}. \qed

\section*{Acknowledgements} This research is supported by NSERC and NSF.


\begin{bibdiv}
\begin{biblist}

\bib{Ahlfors}{book}{
    AUTHOR = {Ahlfors, Lars V.},
     TITLE = {Lectures on quasiconformal mappings},
    SERIES = {University Lecture Series},
    VOLUME = {38},
   EDITION = {Second},
      NOTE = {With supplemental chapters by C. J. Earle, I. Kra, M.
              Shishikura and J. H. Hubbard},
 PUBLISHER = {American Mathematical Society, Providence, RI},
      YEAR = {2006},
     PAGES = {viii+162},
}

\bib{ABGN}{article}{
    AUTHOR = {Angel, Omer},
    AUTHOR = {Barlow, Martin},
    AUTHOR = {Gurel-Gurevich, Ori},
    AUTHOR = {Nachmias, Asaf},
    TITLE = {Boundaries of planar graphs, via circle packings},
    JOURNAL = {Ann. Prob., to appear},
    YEAR = {2015},
}

\bib{BS2}{article}{
    AUTHOR = {Benjamini, Itai},
    AUTHOR = {Schramm, Oded},
     TITLE = {Harmonic functions on planar and almost planar graphs and
              manifolds, via circle packings},
   JOURNAL = {Invent. Math.},
    VOLUME = {126},
      YEAR = {1996},
    NUMBER = {3},
     PAGES = {565--587},
}

\bib{Chelkak}{article}{
    AUTHOR = {Chelkak, Dmitry},
    TITLE = {Robust Discrete Complex Analysis: A Toolbox},
    Journal = {Ann. Prob., to appear},
}

\bib{Grigoryan}{article}{
    AUTHOR = {Grigor'yan, Alexander},
     TITLE = {Analytic and geometric background of recurrence and
              non-explosion of the {B}rownian motion on {R}iemannian
              manifolds},
   JOURNAL = {Bull. Amer. Math. Soc. (N.S.)},
    VOLUME = {36},
      YEAR = {1999},
    NUMBER = {2},
     PAGES = {135--249},
      ISSN = {0273-0979},
}


\bib{HS}{article}{
    AUTHOR = {He, Zheng-Xu},
    AUTHOR = {Schramm, Oded},
     TITLE = {Hyperbolic and parabolic packings},
   JOURNAL = {Discrete Comput. Geom.},
    VOLUME = {14},
      YEAR = {1995},
    NUMBER = {2},
     PAGES = {123--149},
}

\bib{RG}{article}{
    AUTHOR = {Gill, James T.},
    AUTHOR = {Rohde, Steffen},
     TITLE = {On the Riemann surface type of Random Planar Maps},
   JOURNAL = {Revista Mat. Iberoamericana},
   VOLUME = {29},
   YEAR = {2013},
   PAGES = {1071--1090},
}

\bib{Kanai}{article}{
    AUTHOR = {Kanai, Masahiko},
     TITLE = {Rough isometries, and combinatorial approximations of
              geometries of noncompact {R}iemannian manifolds},
   JOURNAL = {J. Math. Soc. Japan},
  FJOURNAL = {Journal of the Mathematical Society of Japan},
    VOLUME = {37},
      YEAR = {1985},
    NUMBER = {3},
     PAGES = {391--413},
}


\bib{K36}{book}{
	Author = {Koebe, Paul},
	Publisher = {Hirzel},
	Title = {Kontaktprobleme der konformen Abbildung},
	Year = {1936}
}

\bib{LP}{book}{
    author = {{R. Lyons with Y. Peres}},
    title = {Probability on Trees and Networks},
    publisher = {Cambridge University Press},
    date = {2008},
    note = {In preparation. Current version available at \texttt{http://mypage.iu.edu/\~{}rdlyons/prbtree/book.pdf}},
}


\bib{RS}{article}{
    AUTHOR = {Rodin, Burt},
    AUTHOR = {Sullivan, Dennis},
     TITLE = {The convergence of circle packings to the {R}iemann mapping},
   JOURNAL = {J. Differential Geom.},
    VOLUME = {26},
      YEAR = {1987},
    NUMBER = {2},
     PAGES = {349--360},
}

\bib{R}{article}{
    AUTHOR = {Rohde, Steffen},
     TITLE = {Oded Schramm: from circle packing to SLE},
   JOURNAL = {Ann. Probab.},
    VOLUME = {39},
      YEAR = {2011},
    NUMBER = {5},
     PAGES = {1621--1667},
}
%

\bib{St}{book}{
    AUTHOR = {Stephenson, Kenneth},
     TITLE = {Introduction to circle packing},
      NOTE = {The theory of discrete analytic functions},
 PUBLISHER = {Cambridge University Press},
   ADDRESS = {Cambridge},
      YEAR = {2005},
     PAGES = {xii+356},
}

%








\end{biblist}
\end{bibdiv}
\end{document}